\documentclass[11pt]{article}
\usepackage{amsmath}
\usepackage{amsfonts}
\usepackage{amssymb}
\usepackage{amsthm}
\usepackage{graphicx}
\usepackage{psfrag}
\usepackage{epsf}
\usepackage{mathrsfs,eucal,xspace}
\usepackage[all]{xy}
\usepackage{hyperref,amssymb,stmaryrd}
\usepackage{pstricks,pst-node}

\newtheorem{theorem}{Theorem}[section]
\newtheorem{proposition}[theorem]{Proposition}
\newtheorem{assumption}[theorem]{Assumption}

\newtheorem{lemma}[theorem]{Lemma}
\newtheorem{corollary}[theorem]{Corollary}
\newtheorem{remark}[theorem]{Remark}

\newtheorem{conjecture*}{Conjecture}
\newtheorem*{properties}{Manifold Properties}

\newcommand{\spam}{\mathop{\mathrm{span}}}

\newcommand{\ints}{\mathbb{Z}}
\newcommand{\nats}{\mathbb{N}}
\newcommand{\reals}{\mathbb{R}}

\newcommand{\comps}{\mathbb{C}}
\newcommand{\M}{\mathbb{M}}

\newcommand{\sphere}{\mathbb{S}^d}
\newcommand{\sph}{\mathbb{S}} 
\newcommand{\Exp}{\operatorname{Exp}}
\newcommand{\man}{\M}
\newcommand{\dif}{\mathrm{d}}

\newcommand{\rad}{ \frac{h}{h_0 \Gamma_1}}

\newcommand{\bft}{\mathbf{t}}

\newcommand{\inj}{\mathrm{r}_\M}
\renewcommand{\d}{\mathrm{dist}}

\newcommand{\sfp}{\mathsf{p}}

\newcommand{\sfG}{\mathbf{G}}

\newcommand{\band}{\mathbf{B}}
\newcommand{\res}{\mathbf{R}}
\newcommand{\inrad}{\mathrm{r}_{\M}}  
\newcommand{\injconst}{\inrad}
\newcommand{\kernelnetwork}{S(\kappa,\Xi)}
\newcommand{\Ind}{\Xi}

\def\cN{{\mathcal N}}

\def\cL{{\mathcal L}}

\def\bt{{\bf t}}
\def\bft{{\bf t}}

\numberwithin{equation}{section}
\title{Kernel Approximation on Manifolds II: The $L_{\infty}$-norm of the $L_2$-projector
}
\author{T. Hangelbroek\thanks{ Department of Mathematics, Texas A\&M
    University, College Station, TX 77843, USA. Research supported
    by NSF Postdoctoral Research Fellowship.}, 
F. J. Narcowich\thanks{ Department of Mathematics, Texas A\&M
    University, College Station, TX 77843, USA. Research
    supported by grant DMS-0807033 from the National
    Science Foundation.}, 
X. Sun\thanks{ Department of Mathematics, Missouri State University Springfield, MO 65804, USA.},
J. D. Ward\thanks{ Department of Mathematics, Texas A\&M University,
    College Station, TX 77843, USA. Research supported by
    grant DMS-0807033 from the National Science
    Foundation.}  }

\begin{document}
\maketitle
\begin{abstract}
This article addresses two topics of significant mathematical and practical
interest in the theory of kernel approximation: the existence of 
local and stable bases and the $L_p$--boundedness of the least squares operator.
The latter is an analogue of the classical problem
in univariate spline theory, known there as the ``de Boor conjecture''.
A corollary of this work is that for appropriate kernels the least squares
projector provides universal near-best approximations for functions $f\in L_p$,
$1\le p\le \infty$.
\end{abstract}

\section{Introduction} 

The purpose of this article is two-fold: to provide operator norm estimates for least squares approximation using linear combinations of translates of kernels, and
to demonstrate that there are local, stable bases for the basic spaces of kernel approximation. 
We consider the problem of finding the best $L_2$ approximant from the finite dimensional space 
$$S(\kappa,\Xi) := \spam_{\xi\in\Xi}\kappa(\cdot,\xi),$$ 
where $\Xi \subset \M$, $\M$ is a $d-$dimensional, compact, complete, smooth Riemannian manifold, and $\kappa\colon \ \M\to \reals$ is a positive definite kernel (we make more assumptions about it, anon). In pursuing this result, we identify a basis $(v_{\xi})_{\xi\in\Xi}$ for  $S(\kappa,\Xi)$ having uniformly bounded condition numbers.

A large body of the literature has been devoted to the \index{}interpolation problem: given $f$,  find $I_{\Xi}f \in S(\kappa,\Xi)$
by enforcing the condition
\begin{equation} \label{nativespace-proj}
 (I_{\Xi} f)|_\Xi  = f|_\Xi.
\end{equation}
The cornerstone of the theory is that 
the linear operator $I_{\Xi}$ 
is the orthogonal projection on a certain reproducing kernel Hilbert space 
$\cN_\kappa$ (called the {\em native space}) to its subspace 
$\kernelnetwork$.
In other words, $I_{\Xi}f$ is the best approximant to $f$ from $S(\kappa,\Xi)$ in $\cN_\kappa$.

Recently, the authors of \cite{HNW} have proved that, under basic geometric
conditions on the sets $\Xi$ 
(a condition defined in the following section known as {\em quasiuniformity}, 
which we assume throughout this article) 
the $L_\infty$ norm of the interpolation operator $I_{\Xi}f$, known as the Lebesgue constant $\cL$, is bounded independent of $\Xi$. 
I.e., the interpolation operators $I_{\Xi}$, 
originally bounded by $1$ in the $\cN_\kappa\to \cN_\kappa$ norm,
are projectors from $C(\M)$ to $S(\kappa,\Xi)$ with 
$L_{\infty}\to L_{\infty}$ norms bounded by a constant independent of the dimension of $S(\kappa,\Xi)$. 
An immediate corollary is that for $f \in C(\M)$,
interpolation is {\em near best} approximation,  as is shown in the following standard inequality:
\[
 \|f - I_{\Xi}f \|_\infty \le (1 + \cL) \d_{\infty}(f, S(\kappa,\Xi)).
\]
To prove this result, a local basis for $\kernelnetwork$ was utilized. In the first part of this article, we investigate the $L_p$ {\em stability} of this basis: i.e., we
give conditions on the kernel sufficient to guarantee norm equivalence of the $L_p$ norm of a function in $S(\kappa,\Xi)$ and the sequence norm of its coefficients (relative to the local basis).

Our second goal is to study the least squares projector $T_\Xi$ from from 
$L_p(\M)$, $1\le p\le\infty$  to  $S(\kappa,\Xi)$. For each $f \in L_p(\M)$, 
we define $T_\Xi f$ to be the unique element in $S(\kappa,\Xi)$ such that
\[
 \int_\M \bigl(f(x)-\bigl(T_\Xi f\bigr) (x)\bigr)  g(x) dx=0, \quad g \in S(\kappa,\Xi).
\]
When $p=2$, $T_{\Xi}f$ is the best $L_2$ approximation to $f\in L_2$ from 
$S(\kappa,\Xi)$. 
Because it is a projector,  
\[
 \|f - T_\Xi(f) \|_p \le (1 + \|T_\Xi \|_p) \d_p(f, S(\kappa,\Xi)).
\]
We show that having a local and stable bases is sufficient to bound uniformly the 
norm of the least square projector (independent of $\#\Xi$, similar to the case of the 
Lebesgue constant). In particular, this guarantees that least squares projection is
near best approximation in $L_p$, $1\le p\le \infty$.

The development here is reminiscent of an episode in the history of the univariate 
B-spline theory, dubbed the ``de Boor conjecture'' \cite{deboorquasi}, which
stated that the $L_\infty-$ norm of the least square projector onto the space of splines 
of order $k$ can be bounded independently of the knot sequence.
Douglas, Dupont, and Wahlbin \cite{DDW} established the $L_\infty$ 
boundedness of the least square projector onto the space of splines of order 
$k \;(k \ge 1)$ with knot sequence $\bt:=(t_j)$ 
in terms of the mesh ratio:
\[
 M_t:= \max_{i} (t_{i+1}-t_i)/\min_{i} (t_{i+1}-t_i).
\]
A more stable upper bound for the projector in terms of the smaller global mesh ratio
\[
 M^{(k)}_t:= \max_{i} (t_{i+k}-t_i)/\min_{i} (t_{i+k}-t_i).
\]
was given by de Boor \cite{deboor1}.
Much later, Shadrin \cite{shadrin} gave an affirmative answer to the conjecture.

Although the manifold setting is rather abstract, conducting our 
investigation in this generality is sufficient to cover many practical applications, including 
approximation on spheres ($\mathbb{S}^d$), projective spaces ($\mathbb{P}^d$), and 
many matrix groups. On the other hand, little is required in the way of geometric 
sophistication; most statements hold for compact metric spaces, and the focus on manifolds 
comes mainly because the examples of local kernel bases from \cite{HNW} have been 
developed in this setting.

We organize this article as follows. In Section~\ref{prep}, we make necessary 
preparation for proving the main results. 
In Section~\ref{Cond}, we establish upper and lower condition numbers for the 
Lagrange basis  under some simple assumptions. 
Section~\ref{LeastSquares} investigates the problem of continuous least squares minimization.
We use the condition numbers obtain in Section~\ref{Cond} to estimate the spectrum of 
the Gram matrix for this problem and bound the $\ell_{\infty}\to \ell_{\infty}$ norm of its inverse:
we show that this matrix is boundedly invertible when the basis is local and stable.
In Section~\ref{Main}, we  state the main results for $L_p$ stability of the $L_2$ projector,
the $L_2$ projector is $L_p$ stable when there is a local, stable basis.
We then observe that this holds for the spaces generated by the kernels introduced in \cite{HNW}.  
An example is given for the 2-sphere where precise approximation orders are known. 
A few conjectures are discussed in Section~\ref{Conclusion}, while the Appendix
houses some technical proofs about the kernels introduced in \cite{HNW}.

\section{Background and Notation} \label{prep}

\subsection{Positive definite kernels}

The motivation for this article was to consider least squares approximation from 
$S(\kappa,\Xi)$, with $\kappa$  a positive definite kernel and $\Xi$ a finite subset of $\M$. 
In this case, a basis can be obtained directly from the kernels,
$v_{\xi} = \kappa(\cdot,\xi)$, $\xi\in \Xi$
and, until recently, it may  have seemed natural to consider this basis for most computational 
problems (the reproducing property of $\kappa$:
 \[
 \langle f (\cdot), \kappa (\cdot, x) \rangle = f (x), \quad f \in \cN_\kappa,\quad  x \in \M,
\]
perhaps motivated this natural instinct).
However, such bases are neither local nor stable, which results in analytical and numerical 
difficulties. In particular, the global nature of the kernel leads to a full Gram matrix, 
$\sfG = \bigl(\langle \kappa(\cdot,\xi), \kappa(\cdot,\zeta)\rangle\bigr)_{\xi,\zeta\in\Xi}$, 
and little can be said about the norm of its inverse. 

Even in the case of the compactly supported Wendland functions, as the number of centers 
increases the bands of associated collocation or Gram matrices become ever larger. 
Instead, we wish to use a basis that is local, and that scales according to the density of the centers. 

Recently, in \cite{HNW}, this has been demonstrated for Lagrange functions associated with interpolation using certain kernels. For a general compact, Riemannian $d-$dimensional 
manifold $\M$, a family of kernels $\kappa_{m,\M},$ for integers $m>d/2$ were considered. 
These kernels have as their native space $\cN_{\kappa_{m,\M}}$ the Sobolev space 
$W_2^m(\M)$.

\subsection{Manifolds}

Throughout this paper, $\M$ denotes a compact, $d$-dimensional Riemannian manifold.
We choose to carry out our investigation in the manifold setting in order to make the results 
applicable to a broad spectrum of practical problems, 
including approximation on spheres ($\mathbb{S}^d$). 
However, familiarity with Riemannian manifolds
is not necessary for a thorough understanding of the arguments we present here.
For the definitions, theorems and proofs in this article, the salient
facts are that $\M$ is a compact metric space, and a few more conditions we
now make clear.
 
We denote the ball in $\M$ centered at $x$ having radius $r$ by $B(x,r).$ 
Given a finite set $\Xi\subset\M$, we define its  \emph{mesh norm}  (or \emph{fill distance}) $h$ 
and the \emph{separation radius} $q$ to be:
\begin{equation} \label{minimal-separation}
 h:=\sup_{x\in \M} \d(x,\Xi)\qquad \text{and}\qquad  q:=\frac12 
\inf_{\xi,\zeta\in \Xi, \xi\ne \zeta} 
\d(\xi,\zeta).
\end{equation}
The mesh norm measures the density of $\Xi$ in $\M$, the separation radius determines 
the spacing of $\Xi$. The \emph{mesh ratio} $\rho:=h/q$ measures the uniformity of the 
distribution of $\Xi$ in $\M$. 
We say that the point set $\Xi$ is quasi-uniformly distributed, or simply that $\Xi$ is quasi-uniform if 
$\Xi$ belongs to a class of finite subsets with mesh ratio bounded by a constant $\rho_0$.\footnote{Naturally,
every finite point set has a bounded mesh ratio. Our results hold for point sets from a (not explicitly
defined). 
Inclusion in this family means that the mesh ratio
is bounded by an unspecified constant whose size influences various constants in our results.}

The manifold is endowed with a measure, and we indicate the measure of
subsets $\Omega\subset \M$ by $\mathrm{vol}(\Omega)$. The integral, and the $L_p$ 
spaces for $1\le p\le \infty$, are defined with respect to this measure.  The embeddings
$$C(\M) \subset L_p(\M)\ \text{for}\ 1\le p \le \infty 
\quad \text{and} \quad 
L_p(\M) \subset L_q(\M)\ \text{for}\ 1\le q\le p\le \infty$$ 
hold. In addition, $L_2$ is a Hilbert space equipped with the inner product 
$\langle\, \cdot\, ,\, \cdot \,\rangle\colon \ (f,g)\mapsto \langle f,g\rangle$

The final properties of the manifold, in force throughout the article, are the following.
\begin{properties}\label{Man} 
There exist constants 
$0<\alpha_{\M},K_{\M}, \omega_{\M}<\infty$ and  $0<\inrad< \infty$ for which the 
following hold:
\begin{enumerate}
\item 
for all $x\in \M$  and all $r\le \inrad$ , 
$$\alpha_{\M} r^d \le \mathrm{vol}(B(x,r))\le \omega_{\M} r^d;$$
\item 
for $\Xi\subset \M$, $x\in\M$ and $R\ge q$,
 $$\#\bigl(\Xi\cap B(x,R)\bigr) \le K_{\M} (R/q)^d.$$
\end{enumerate}
\end{properties}

We note that the preceding holds when 
$\inrad$ is the injectivity radius, which 
bounds the radii of balls in which local geometric properties of the manifold are similar to those of $\reals^d$.\footnote{This is the radius of the largest ball in $\reals^d$ for which the exponential map (centered at any point of the manifold) is a diffeomorphism between the ball and the corresponding neighborhood in the manifold.} Its technical definition is unimportant for our purposes here, although it appears in the statement of a key condition (Assumption \ref{LagrangeDecay}), as well as in the definition of the previous constants.

%
\section{Condition of the Lagrange Basis }\label{Cond}

In this section we show that a Lagrange basis is stable if it is local and satisfies an equicontinuity condition.

Let $(v_{\xi})_{\xi\in\Xi}$ be a linearly independent family of continuous functions  on $\M$, 
indexed by a finite set $\Xi\subset \M$.  In this section we present conditions sufficient to 
establish an equivalence between the $L_p$ norm of 
$s \in X:=\spam_{\xi\in\Xi} v_{\xi}$ 
and the $\ell_p$ norm of its normalized coefficients.  
Thus, we seek constants $0 < c_1 < c_2$, called {\em condition numbers}, so that 
for $s = \sum_{\xi\in \Xi} A_{\xi} v_{\xi}$, the two sided estimate 
\begin{equation} \label{comparison}
c_1q^{d/p}\bigl( \sum_{\xi\in\Xi}|A_{\xi}|^p\bigr)^{1/p}
\le 
\|\sum_{\xi \in \Xi} A_{\xi} v_{\xi}  \|_{L_p(\M)}
\le 
c_2q^{d/p}\bigl( \sum_{\xi\in\Xi}|A_{\xi}|^p\bigr)^{1/p}
\end{equation}
holds true, with the usual modification, 
$\max_{\xi\in\Xi} |A_{\xi}|\sim \|\sum_{\xi \in \Xi} A_{\xi} v_{\xi} \|_{\infty}$,
when $p=\infty$. 
In particular the basis to be considered will be the Lagrange basis (see Assumption~\ref{LagrangeBasis}). 
One observes that, as expected, the condition numbers depend in a controlled way on $p$, as
is the case with B-spline bases and wavelet bases, and this dependence on $p$ (encoded in the factors $q^{d/p}$)
can be suppressed by taking an appropriate renormalization of the basis
(see Remark~\ref{renorm} below).

Our ultimate goal is to prove bounds (\ref{comparison}) for certain kernel spaces 
$X=S(\kappa,\Xi)$ with condition numbers that are independent of $\# \Xi$. 
However, the development of this section does not make use of the fact that $X$ 
is a kernel space (and there may be instances where one would not wish to assume 
this). Instead, we prefer to state our results for a general space, $X$, in terms of a few 
basic assumptions on the the basis functions $v_{\xi}$, and then to show that such 
assumptions hold for the kernel spaces we consider. 
\begin{remark}\label{Nikolskii}
An immediate consequence of (\ref{comparison}) is the Nikolskii inequality, 
which states that for $1\le p\le q\le \infty$, \begin{equation}\label{eqNikolskii}
\|s\|_q \le \frac{c_2}{c_1} q^{-d(\frac1{p} - \frac1{q})} \|s\|_p\quad \text{for all}\ s\in X.
\end{equation} 
In case $X=S(\kappa,\Xi)$, the penalty $q^{-d(\frac1{p} - \frac1{q})_{+}} $ is in line with 
other examples
of Nikolskii's inequality, since $\dim(S(\kappa,\Xi)) \sim q^{-d}$ when $\kappa$ is positive 
definite and $\Xi$ is quasiuniform.
\end{remark}
\begin{remark}\label{renorm}
We often write $s \in X$ as a renormalized series
$s(x)=\sum_{\xi \in \Xi} A_{p,\xi} v_{p,\xi} (x),$
where $v_{p,\xi}(x) := q^{-d/p}v_{\xi}(x)$ and  $A_{p,\xi} := q^{d/p} A_{\xi}$. 
\end{remark}
In establishing condition numbers, we employ three basic assumptions on the 
family of functions $(v_{\xi})_{\xi\in\Xi}$. 
We assume the basis consists of the Lagrange functions, 
the family is fairly localized (exhibiting rapid decay away from a center), 
and the basis functions are H\"{o}lder continuous, with a uniform seminorm that 
scales appropriately with the centers $\Xi$.  We now formally state these. 
\begin{assumption}[Lagrange Basis]\label{LagrangeBasis}
The basis $(v_{\xi})_{\xi\in\Xi}$ is the  Lagrange basis. I.e.,
$$v_{\xi} (\zeta) = \delta(\xi,\zeta)\quad\text{for}\ \xi,\zeta\in\Xi.$$
\end{assumption}
A consequence of this assumption is that coefficients are simply the sampled function 
values. Thus, the inequality we are after can be recast as a {\em sampling} inequality 
\begin{equation}\label{sampling}
c_1\|  s|_{\Xi}\|_{\ell_p(\Xi)} 
\le 
q^{-d/p} \left\|s \right\|_{L_p(\M)}
\le 
c_2 \| s|_{\Xi}\|_{\ell_p(\Xi)}. 
\end{equation}
Literature abounds in the establishment of similar inequalities. In particular, 
if the functions considered are polynomials, then these are often referred to 
as Marcinkiewicz--Zygmund inequalities (\cite{MZ}, \cite{filbir}, \cite{schmid}); 
if the functions considered are entire functions of exponential type, then these 
are referred to as Plancherel-P\'{o}lya type inequalities 
(\cite{P-P}, \cite{pesenson07}, \cite{pesenson09}).

The second assumption is a localization assumption, and is stated for a general 
family of functions (not necessarily a basis) indexed by the centers $\xi\in\Xi$. 
\begin{assumption}[Localization] \label{LagrangeDecay} 
For a family of functions $(v_{\xi})_{\xi\in \Xi}$,
 there exist positive constants $\nu$ and $C_1$
for which each individual function
$v_{\xi}\colon \ \M\to \comps$ satisfies
$$
|v_{\xi}(x)|
\le 
C_1 \exp\left(-\nu \frac{\min\bigl(\d(x,\xi),\injconst\bigr)}{q}\right).
$$ 
\end{assumption}
Under the assumption that the set $\Xi$ is  quasi-uniform, the following is 
equivalent to Assumption \ref{LagrangeDecay}:
\begin{equation}\label{LDecay2}
|v_{\xi}(x)|
\le 
C_1 \exp\left(-\nu' \frac{\min\bigl(\d(x,\xi),\injconst\bigr)}{h}\right).
\end{equation}
with $\nu' \rho= \nu$.  We will use them interchangeably according to convenience. 

Similar bounds have been established in \cite{HNW} for the Lagrange functions obtained
from a  large class of kernels. These are the kernels $\kappa_{m,\M}$, $m>d/2$, 
 the  reproducing kernels for the Sobolev space $W_2^m(\M)$ endowed with a 
certain inner-product.\footnote{More precisely, they are the reproducing kernels for the Sobolev spaces $W_2^m(\M)$ endowed with the inner product $\langle u,v \rangle_m = \sum_{k=0}^m\int_{\M} \langle \nabla^k u, \nabla^k v\rangle_x \dif x$, where $\nabla$ is the covariant derivative on $\M$, which maps tensor fields of rank $j$ to tensor fields of rank $j+1$. Although this is of fundamental importance to establishing the localness of the Lagrange basis (and
for bounding the Lebesgue constant) it is a technicality which does not directly play a major role here. We direct the interested reader to \cite[Section 2]{HNW} for details.}
An immediate corollary is that the Lebesgue constant $\sup_{x\in \M} \sum_{\xi\in\Xi}|\chi_{\xi}(x)|$ is bounded in terms of $C_1$, $\nu$, $\rho$ and $\M$ (cf. \cite[Theorem 4.6]{HNW}).

The final assumption is one of equicontinuity. 
\begin{assumption}[H\"{o}lder Continuity] \label{Bernstein} There is $0<\epsilon \le 1$ so  
that for any quasiuniform set $\Xi\subset \M$ with separation distance $0<q$,
$$
| v_{\xi}(x) - v_{\xi}(y)|
\le 
C_2 \left[\frac{\d(x,y)}{q}\right]^{\epsilon} 
$$ 
with constant $C_2$ depending only on $\epsilon, \M$ and the mesh ratio $\rho = h/q$ 
(but not on $\#\Xi$).
\end{assumption}
This final assumption is closely related to Bernstein inequalities, which have been investigated recently  on spheres in \cite{MNPW}, 
and have been demonstrated for a very large class of kernels. 
Namely, for all zonal kernels that are a perturbation
of the fundamental solution of an elliptic partial differential operator by
a convolution of the solution with an $L_1$ function (this is \cite[Theorem 5.5]{MNPW}). 
We remark that \cite[Section 5]{HNW} demonstrates that the kernel $\kappa_{2,\mathbb{S}^2}$ of order 2 on $\mathbb{S}^2$ is such a perturbation.
Using entirely different techniques, we will demonstrate that all kernels of the form 
$\kappa_{m,\M}$ satisfy Assumption \ref{Bernstein} in Lemma \ref{Bern_lemma} in the appendix.
\subsection{Upper Bounds}
%
%
We depart temporarily from using Lagrange functions to prove the upper bound,
since this property is not necessary. 
We treat families of functions  that merely satisfy the localization property, 
Assumption \ref{LagrangeDecay}, with centers that are quasiuniform.
\begin{proposition}\label{uppercomparison}
For quasiuniform centers $\Xi\subset \M$ and a family of functions
$(v_{\xi})_{\xi\in \Xi}$ satisfying Assumption \ref{LagrangeDecay},
we write $s = \sum_{\xi\in\Xi} A_{\xi} v_{\xi} = \sum_{\xi\in\Xi} A_{p,\xi} v_{p,\xi}$.
There exists a constant $c_2$, depending only on $\M$, $\rho$  and the constants 
appearing in Assumption \ref{LagrangeDecay} so that,
$$  \|s\|_p \le c_2\| A_{p,\cdot}\|_{\ell_p(\Xi)}.$$
\end{proposition}
\begin{proof}
For $p=\infty$, the upper bound follows by a standard argument, decomposing the 
sum {\em en annuli}. 
This is the argument used to bound the Lebesgue constant in the proof of 
\cite[Theorem 4.6]{HNW} (in fact, if $v_{\xi}$ is the Lagrange basis, the upper 
bound we seek  is exactly the Lebesgue constant). Thus we have $\mathcal{L}$ for 
which
$$
\|s\|_{L_{\infty}(\M)}\le \mathcal{L}\|s|_{\Xi}\|_{\ell_{\infty}(\Xi)} 
=
\mathcal{L}\|A_{\infty,\cdot}\|_{\ell_{\infty}(\Xi)},
$$
For $p=1$, we have
$$
\int_{\M} |s(x)|\dif x 
\le 
C \|A_{1,\cdot}\|_{\ell_{1}(\Xi)}.
$$
Here we have used the fact that $\|v_{1,\xi}\|_1 \le C$ for some constant $C$ depending 
only on the manifold $\M$ and the constants in Assumption \ref{LagrangeDecay}.
A standard application of operator interpolation proves the other cases. The constant 
$c_2$ can be taken to be $C^{1/p} \mathcal{L}^{1-1/p}$.
\end{proof}
Note that in the proof above we did not use Assumption \ref{LagrangeDecay} in its 
full force. In fact, the following much less stringent decay condition
for the family of functions 
\[
|v_{p,\xi}(x)|\lesssim q^{-d/p}(1+\d(x,\xi)/q)^{-\mu}, \quad \mu>d,
\]
is sufficient for Proposition \ref{uppercomparison} to hold true.
%
%
%
%
%
\subsection{Lower bounds}
The lower bounds in this case are {\em sampling} inequalities, and the lower bound in 
the $L_{\infty}$ case follows immediately with constant $1$. The main challenge is to 
establish the $L_p$ lower comparison, and, in contrast with the upper bound, there is 
no convenient way to make use of Riesz--Thorin or other interpolation results. In our proof,  
it is not just the decay of the basis 
elements, but also their nature as Lagrange functions and their equicontinuity that is 
important. 
%
%
\begin{proposition}\label{lowercomparison}
Under Assumptions \ref{LagrangeBasis}, \ref{LagrangeDecay} and \ref{Bernstein},
there exist constants $c_1>0$ and $q_0>0$, so that for $q<q_0$, 
for $1\le p\le \infty$ and  for all $s\in X$,
$$
c_1 \| s|_{\Xi}  \|_{\ell_p(\Xi)}       
\le  
q^{-d/p} \|s\|_p.
$$
The constants $c_1$ and $q_0$ depend only on $\rho$, the constants appearing in Assumptions \ref{LagrangeDecay} and \ref{Bernstein} (namely $C_1$, $\nu$, $C_2$, $\epsilon$),  and on $\M$ (specifically on the constants appearing in Assumption \ref{Man}).
\end{proposition}
\begin{proof}
The $L_{\infty}$ case follows immediately with constant $1$. We will therefore assume 
in the rest of the proof that $1 \le p < \infty$.

We estimate the $L_p$ norm of the function by integrals over a union of disjoint balls 
$\{B(\xi,\gamma q)\colon \  \xi \in \Xi\} $ with $\gamma\le 1$:
$$
\|s\|_p^p 
= 
\int_{\M}\Big|\sum_{\zeta\in\Xi} s(\zeta) v_{\zeta}(x)\Big|^p \dif x 
\ge 
\sum_{\xi\in \Xi}
  \int_{B(\xi,\gamma q)}
    \Big|\sum_{\zeta\in\Xi} s(\zeta) v_{\zeta}(x)\Big|^p 
  \dif x .
$$
The H\"{o}lder continuity assumption implies that, for $\gamma>0$ 
sufficiently small $|v_{\xi}(x)|\ge \frac{2}{3}$ in a neighborhood $B(\xi,\gamma q)$ 
(the choice of $\gamma^{\epsilon} \le1/(3 C_2)$ does the trick). 
Consequently, we have, for each fixed $\xi\in\Xi$, that
$$
\int_{B(\xi,\gamma q)}  |v_{\xi}(x)|^p\dif x 
\ge
\min_{x\in B(\xi,\gamma q)}   |v_{\xi}(x)|^p
\times \bigl(\mathrm{vol}\bigl(B(\xi,\gamma q)\bigr)\bigr) 
\ge
\left(\frac{2}{3}\right)^p \alpha_{\M}  (\gamma q)^d.
$$
We now use a diagonal dominance argument: 
we show that, for each fixed $\xi \in \Xi$, the integral 
$\int_{B(\xi,\gamma q)}
    |\sum_{\zeta\in\Xi} s(\zeta) v_{\zeta}(x)|^p 
  \dif x $ 
is dominated by the ``principal term'' 
$\int_{B(\xi,\gamma q)}  |s(\xi) v_{\xi}(x)|^p  \dif x $. 
To show this, we estimate the contribution of the {\em off-center} terms.
In short, we have the estimate
$$
\sum_{\xi\in\Xi}
  \int_{B(\xi,\gamma q)}
    \Big|\sum_{\zeta\in\Xi} s(\zeta) v_{\zeta}(x)\Big|^p 
  \dif x
  \ge  
\sum_{\xi\in\Xi} 
  \left( 
    2^{1-p} \int_{B(\xi, \gamma q)}| s(\xi) v_{\xi}(x)|^p \dif x
    - 
    \int_{B(\xi, \gamma q)}
      \Big|\sum_{\zeta\neq \xi} s(\zeta) v_{\zeta}(x)\Big|^p 
    \dif x 
  \right).\nonumber
  $$
This is a consequence of the following inequality which we use at times throughout the paper, and state here formally: 
\begin{equation}\label{finite_holder}\left|\sum_{j=1}^n a_j\right|^p\le n^{p-1}\sum_{j=1}^n |a_j|^p.
\end{equation}
In this case we have applied (\ref{finite_holder}) with $n=2$, $a_1 = \sum_{\zeta\in\Xi} |s(\zeta) v_{\zeta}(x)|$
and $a_2 = - \sum_{\zeta\in\Xi\setminus\xi} |s(\zeta) v_{\zeta}(x)|$, which after a manipulation
gives $ |a_1|^p\ge 2^{1-p}(a_1+a_2)^p - |a_2|^p $.
Combining this with the above estimate of the principal term gives
\begin{equation}
\|s\|_p^p
\ge
\sum_{\xi\in\Xi}  
  \left(
    \frac{2\alpha_{\M}}{3^{p}}  (\gamma q)^d |s_{\xi}|^p 
    -  
    \int_{B(\xi,\gamma q)}
      \Big|\sum_{\zeta\neq \xi} s(\zeta) v_{\zeta}(x)\Big|^p 
    \dif x
  \right). \label{first_split}
\end{equation}

By shrinking $\gamma$, the sum  $\sum_{\xi\in\Xi} \int_{B(\xi,\gamma q)}
      |\sum_{\zeta\neq \xi} s(\zeta) v_{\zeta}(x)|^p 
    \dif x$ can be made as small as desired.
In particular as shown in Lemma~\ref{offcenter} below, the off-center sum can be made less than one half of the principal part of \eqref{first_split}, $\sum_{\xi\in\Xi}\frac{2\alpha_{\M}}{3^{p}} (\gamma q)^d |s_\xi|^p$, and the choice of $\gamma$ depends only on $\M$ and the constants from Assumptions~\ref{LagrangeDecay} and \ref{Bernstein}, and not all on $p$. As a consequence
\[
 \|s\|^p_p \ge \frac{\alpha_\M}{3^p} (\gamma q)^d \sum_{\xi\in \Xi} |s(\xi)|^p.
\]
Thus
\begin{align}
 c_1 &= \min_{p\in [1,\infty]} (\alpha_\M\gamma^d)^{1/p}/3\notag\\
&= \min(\alpha_\M \gamma^d/3,1).\tag*{$\qed$}
\end{align}
\renewcommand{\qed}{}\end{proof}


\begin{lemma}\label{offcenter}
There exist two constants $q_0$ and $R$, depending only on the manifold 
$\M$ and the constants from Assumptions \ref{LagrangeDecay} and \ref{Bernstein}, 
such that the inequality
$$
\sum_{\xi\in\Xi} 
  \int_{B(\xi,\gamma q)}
      \Big|\sum_{\zeta\neq \xi} s(\zeta) v_{\zeta}(x)\Big|^p 
   \dif x
\le
\frac{ \alpha_{\M}}{ 3^{p}}(\gamma q)^d \sum_{\xi\in\Xi}  |s(\xi)|^p
$$
holds for all $s\in X$ and all $p\in [1,\infty)$ when $q < q_0$ and $\gamma < R$.
\end{lemma}
\begin{proof}
Given $\gamma \le 1$, $\Gamma \ge 1$, $q\le \Gamma/\injconst$ and $\xi \in \Xi$, we write
$$  \int_{B(\xi,\gamma q)}\Big|\sum_{\zeta\neq \xi} s(\zeta) v_{\zeta}(x)\Big|^p \dif x 
\le 
 (I_{\xi}+II_{\xi}+III_{\xi}),$$
where, using (\ref{finite_holder}) with $n=3$, we have
 \begin{align*}
I_{\xi}
&:=
3^{p-1}
\int_{B(\xi,\gamma q)}
  \left| \sum_{\substack{\zeta\neq \xi \\ \d(\zeta, \xi)\le \Gamma q}} s(\zeta) v_{\zeta}(x)\right|^p 
\dif x \\
II_{\xi} 
&:= 
3^{p-1} 
\int_{B(\xi,\gamma q)}
  \left|\sum_{\Gamma q<\d(\zeta,\xi)\le \injconst} s(\zeta) v_{\zeta}(x)\right|^p 
\dif x \\
III_{\xi}
&:= 
3^{p-1}
\int_{B(\xi,\gamma q)}
  \left|\sum_{\d(\zeta,\xi)> \injconst} s(\zeta) v_{\zeta}(x)\right|^p 
\dif x.
\end{align*}
We will bound the three sums 
$\sum_{\xi \in \Xi}I_{\xi}, \ \sum_{\xi \in \Xi}II_{\xi},$ and $\sum_{\xi \in \Xi}III_{\xi}$
so that each is less than $\frac{1}{3} \frac{ \alpha_{\M}}{ 3^{p}}(\gamma q)^d \sum_{\xi\in\Xi}  |s(\xi)|^p$.

To obtain the desired estimate, we adopt the following strategy:
\begin{enumerate} 
\item we first choose $\Gamma$ sufficiently large to bound $\sum II_{\xi}={o}\bigl(\Gamma^{-1} \bigr)\times(\gamma q)^d \sum_{\xi\in\Xi}  |s(\xi)|^p$, 
\item  we then choose $q$ sufficiently small to bound $\sum III_{\xi}=o\bigl(q\bigr) \times(\gamma q)^d \sum_{\xi\in\Xi}  |s(\xi)|^p$, at the same time keeping $q<\injconst/\Gamma$, 
\item we  finally show that $\sum I_{\xi}\le (\mathcal{K} \Gamma^{d} \gamma^{\epsilon })^p (\gamma q)^d \sum_{\xi\in\Xi}  |s(\xi)|^p$ 
where  $\mathcal{K}$ is a constant depending only on $\M$ and the constants from Assumptions \ref{LagrangeDecay} and \ref{Bernstein}, which allows us to choose $\gamma$.
\end{enumerate}
The order in which we carry this out is important, because both $q$ and $\gamma$ depend on 
$\Gamma$, and we wish to be able to choose them independently of $p$.

{\bf Step 1, estimating $II$:} The quantity $II_{\xi}$ is subdivided into 
(comparable to) $|\log q|$-many terms, 
each term being a sum over a dyadic region:
$$
\Omega_k
:=
\Omega_k(\xi)
:=
\{
  \zeta\in\Xi 
  \mid 
  \Gamma 2^k q \le \d(\xi,\zeta) \le \Gamma 2^{k+1} q
\},\quad k=0, 1, \ldots, N_{q},
$$ 
in which
$2^{N_q}\sim \frac{\injconst}{\Gamma q}$. 
This means that, for
$M_k
:=
\int_{B(\xi, \gamma q)}
  \left|
    \sum_{\zeta\in \Omega_k} 
      s(\zeta) v_{\zeta}(x)
  \right|^p 
\dif x, 
$
$$
II_{\xi} 
\le  
(3)^{p-1}\sum_{k=0}^{N_{q} }
2^{(p-1)(k+1)}
  \int_{B(\xi,\gamma q)}
    \left|\sum_{\zeta \in \Omega_k} s(\zeta) v_{\zeta}(x)\right|^p 
  \dif x 
= 
(3)^{p-1}\sum_{k=0}^{N_q} 2^{(p-1)(k+1)} M_k,
$$
where the above inequality follows from
$
 \left|\sum^n_{j=1} a_j\right|^p \le \sum^n_{j=1} 2^{j(p-1)} |a_j|^p.
$
This follows easily from repeated application of (\ref{finite_holder}) with $n=2$. We note
that  (\ref{finite_holder})
could be applied directly to estimate $II_{\xi}$, but at a cost 
of incurring a factor $N_q^{p-1}$, which depends on $q$.

We now estimate the contribution from each $M_k$, the portion of $II_{\xi}$ 
coming from the dyadic interval $\Omega_k$.
By using the inequality $ |\sum_{j=1}^n a_j|^p\le n^{p-1}\sum |a_j|^p$, we have
\begin{align}
M_k &\le 
\left(\# \Omega_k\right)^{p-1} \times
\sum_{\zeta\in  \Omega_k} \int_{B(\xi,\gamma q)}|s(\zeta)v_{\zeta}(x)|^p \dif x\nonumber\\
&\le
\left(\# \Omega_k\right)^{p-1}\times \max_{\zeta\in\Omega_k} \|v_{\zeta}\|_{L_{\infty}(B(\xi,\gamma q))}^p
\times \mathrm{vol}(B(\xi,\gamma q))\times \sum_{\zeta\in\Omega_k} |s(\zeta)|^p \nonumber\\
&\le
\left(K_{\M}(2^{k+1}\Gamma)^d\right)^{p-1} \times
\left(C_1^p  \exp(-\nu p\Gamma  2^k) \right)\times
\left(\omega_{\M} (\gamma q)^d \right)\times
\sum_{\zeta\in \Omega_k}|s(\zeta)|^p.\label{holder_trick}
\end{align}
In the final line, we have used the estimate $\#\Omega_k\le K_{\M} (2^{k+1} \Gamma)^d$,
and Assumption \ref{LagrangeDecay}, namely that for
 $\zeta \in \Omega_k$ and $x\in B(\xi,\gamma q)$, we have
$|v_{\zeta} (x)|\le C_1 \exp(-\nu \Gamma  2^k)$. 
Multiplying by $3^{p-1}2^{(p-1)(k+1)}$ and summing from $0$ to $N_q$, we obtain
(after rearranging some terms)
\begin{equation}
II_{\xi}
\le
C_1 \omega_{\M}
(3 K_{\M} C_12^{d+1})^{p-1}  
(\gamma q)^d 
\left( 
  \sum_{k=0}^{N_q}  
    (2^{k(d+1)}\Gamma^{d})^{p-1}
    \exp(-\nu p \Gamma  2^k)  
     \left[
       \sum_{\zeta\in \Omega_k}|s(\zeta)|^p
     \right] 
\right)
\label{TWO}
\end{equation}
We introduce the notation $\mathcal{K}_{II}:= 3 K_{\M} C_1 2^{d+1}$.
Summing the $II_{\xi}$ over $\xi \in \Xi$ and  using the inequality 
$\#\{\xi\colon \ \zeta\in \Omega_k(\xi)\}\le K_{\M} (2^{k+1} \Gamma)^d$ 
(this follows from the fact that, for a given $\zeta\in \Xi$, 
$\#\{\xi\colon \ \zeta\in \Omega_k(\xi)\} = \# \Omega_k(\zeta)$, since
the condition defining these sets is symmetric in $\xi$ and $\zeta$) 
we obtain,
\begin{align*}
\sum_{\xi\in \Xi} II_{\xi}&\le
C_1 \omega_{\M}
\mathcal{K}_{II}^{p-1}
 (\gamma q)^d 
\left( 
  \sum_{k=0}^{N_q} 
    (2^{k(d+1)}\Gamma^d)^{p-1} 
    \exp(-\nu p \Gamma  2^k)  
    \sum_{\xi\in\Xi}
     \left[
       \sum_{\zeta\in \Omega_k}|s(\zeta)|^p
     \right] 
\right)\\
&\le
C_1 \omega_{\M}
\mathcal{K}_{II}^{p-1}
 (\gamma q)^d 
\left( 
  \sum_{k=0}^{N_q} 
    (2^{k(d+1)}\Gamma^d)^{p-1} 
    \exp(-\nu p \Gamma  2^k)
    (K_{\M} (2^{k+1}\Gamma)^{d})  
     \left[
       \sum_{\zeta\in \Xi}|s(\zeta)|^p
     \right] 
\right)\\
&\le
\frac{\omega_{\M}}{3}
\mathcal{K}_{II}^{p} 
\frac{(\gamma q)^d}{\Gamma^p}
\left( 
  \sum_{k=0}^{N_q}  
    (2^{k}\Gamma)^{(d+1)p} 
    \exp(-\nu p \Gamma  2^k)  
    \right) \left[
     \sum_{\zeta\in\Xi}
         |s(\zeta)|^p
     \right].
 \end{align*}
 In the final inequality, we have used the fact that $2^{(k+1)d} \le 2^{k(d+1)}\times 2^{d+1}$
 and that $\Gamma^{dp} = \frac{\Gamma^{(d+1)p}}{\Gamma^p}$. 
 After some minor cleaning up, we have
$$
\sum_{\xi\in \Xi} II_{\xi}
\le
\frac{\omega_{\M}}{3}
\mathcal{K}_{II}^{p} 
\left( 
  \Gamma^{dp}\exp(-\nu p \Gamma) +
  \frac{2}{\Gamma^{p}}
  \int_{\Gamma}^{\infty}  
    \exp\bigl(-\nu p r\bigr)  
    r^{(d+1)p-1} 
  \dif r 
\right)
  (\gamma q)^d
     \|s_{|_{\Xi}}\|_{\ell_p(\Xi)}^p.$$
This estimate holds because $r\mapsto r^{p(d+1)} \exp(-\nu p r)$ is decreasing for $r\ge (d+1)/\nu$.
Because we would like to bound this final expression by $\frac{\alpha_{\M}}{3^{p+1}} (\gamma q)^d
     \|s_{|_{\Xi}}\|_{\ell_p(\Xi)}^p$, it suffices to take
$$
\max\left( \Gamma^{dp}\exp(-\nu p \Gamma),
  \frac{2}{\Gamma^{p}}
  \int_{\Gamma}^{\infty}  
    \exp\bigl(-\nu p r\bigr)  
    r^{(d+1)p-1} 
  \dif r 
\right)
\le 
\frac{\alpha_{\M}}{2\omega_{\M}} (9 C_1 K_{\M} 2^{d+1})^{-p}.
$$
As we demonstrate in Lemma \ref{Inc_Gamma}, using $\epsilon^{-1} = \frac{2\omega_{\M}}{\alpha_{\M}}9 C_1 K_{\M} 2^{d+1}$, this can be accomplished for a choice of $\Gamma$ depending only on $\M$ and the
constants from Assumption \ref{LagrangeDecay}.

{\bf Step 2, estimating $III$:} 
Mimicking the estimate (\ref{holder_trick}) we have
\begin{align*}
\mathbin{III}_{\xi}&\le 
3^{p-1}\left(\# \{\zeta\in \Xi \mid \d(\zeta,\xi)>\injconst\} \right)^{p-1} \times
\sum_{\d(\zeta,\xi)>\injconst} \int_{B(\xi,\gamma q)}|s(\zeta)v_{\zeta}(x)|^p \dif x\\
&\le
3^{p-1}\left(\# \Xi \right)^{p-1}\times \max_{\d(\zeta,\xi)>\injconst} \|v_{\zeta}\|_{L_{\infty}(B(\xi,\gamma q))}^p
\times \mathrm{vol}(B(\xi,\gamma q))\times \sum_{\d(\zeta,\xi)>\injconst} |s(\zeta)|^p
\end{align*} 
It is easy to estimate the cardinality of $\Xi$, since 
$\mathrm{vol}(\M)\ge \sum_{\xi\in\Xi} \mathrm{vol}(B(\xi,q))$. 
By using Assumption \ref{LagrangeDecay} to estimate 
$\max_{\{\zeta\mid\d(\zeta,\xi)>\injconst\}} \|v_{\zeta}\|_{L_{\infty}(B(\xi,\gamma q))}^p$, we obtain
\begin{align*}
\mathbin{III}_{\xi}&\le
3^{p-1}\left(\frac{\mathrm{vol}(\M)}{\alpha_{\M} q^{d}} \right)^{p-1}
C_1^p\exp\left(-\nu p \frac{\injconst}{q}\right) \omega_{\M}  (\gamma q)^d 
\left[
  \sum_{\d(\zeta,\xi)> \injconst}
    |s(\zeta)|^p
\right]
\nonumber\\
&=
C_1\omega_{\M}
\mathcal{K}_{III}^{p-1}
q^{-d(p-1)}
\exp\left(-\nu p \frac{\injconst}{q}\right)
 (\gamma q)^d 
\left[
  \sum_{\d(\zeta,\xi)> \injconst}
    |s(\zeta)|^p
\right],
\end{align*}
with $\mathcal{K}_{III} := 3 C_1 \mathrm{vol}(\M)/\alpha_{\M}. $
Summing over $\xi \in \Xi$, we have
\begin{align*}
\sum_{\xi\in\Xi} \mathbin{III}_{\xi} &\le
C_1\omega_{\M}
\mathcal{K}_{III}^{p-1}
q^{-d(p-1)}
\exp\left(-\nu p \frac{\injconst}{q}\right)
 (\gamma q)^d 
 \sum_{\xi\in\Xi}
\left[
  \sum_{\d(\zeta,\xi)> \injconst}
    |s(\zeta)|^p
\right]\\
&\le
C_1\omega_{\M}
\mathcal{K}_{III}^{p-1}
q^{-d(p-1)}
\exp\left(-\nu p \frac{\injconst}{q}\right)
 (\gamma q)^d 
 \left[
  \sum_{\zeta\in \Xi}
 \left(\frac{\mathrm{vol}(\M)}{\alpha_{\M} q^{d}} \right)
    |s(\zeta)|^p
\right]\\
&=
\frac{\omega_{\M}}{3}
\mathcal{K}_{III}^{p}
q^{-dp}
\exp\left(-\nu p \frac{\injconst}{q}\right)
 (\gamma q)^d 
\left[
  \sum_{\zeta\in \Xi}
    |s(\zeta)|^p
\right].
\end{align*}
We wish to ensure that $\sum_{\xi\in\Xi} III_{\xi}\le \frac{1}{3}\frac{ \alpha_{\M}}{ 3^{p}}(\gamma q)^d \sum_{\xi\in\Xi}  |s(\xi)|^p$, so it suffices to take $q\le q_0$, where
$q_0\le \inj/\Gamma$ is the largest value for which the condition
 $0\le q\le q_0$ implies 
$q^{-d} \exp\left(-\nu r_\M/q\right)
\le \frac{\alpha^2_\M}{\omega_\M} (9C_1 \text{ vol}(\M))^{-1}$.

{\bf Step 3, estimating $I$:} To estimate $I_{\xi}$, we make use of  bounds on $v_{\zeta}$  derived from the H{\"o}lder continuity assumption and the fact that the Lagrange function vanishes at $\xi$, when
$\xi \ne \zeta$, to get
\begin{align*}
I_{\xi} 
&\le 
3^{p-1}\int_{B(\xi, \gamma q)} 
  (K_{\M} \Gamma^d )^{p-1} (C_2 \gamma^{\epsilon})^p  
  \sum_{\d(\zeta,\xi) \le \Gamma q} 
    |s(\zeta)|^p 
\dif x \\
 &\le 
\frac{1}{3}3^{p}K_{\M}^{p-1}C_2^p 
 \Gamma^{d(p-1)} \gamma^{\epsilon p} 
  \omega_{\M} (\gamma q)^d 
 \sum_{\d(\zeta,\xi)\le \Gamma q}   |s(\zeta)|^p.
 \end{align*}
In the first inequality, we use the bound
\[
 |v_{\zeta}(x)|\le C_2 \gamma^{\epsilon}, \quad \text{for $x$ satisfying}\quad \d (x,\xi) \le \gamma q.
\]
and the estimate on the number of centers
\[
 \#\bigl(\Xi\cap B(\xi, \Gamma q)\bigr) \le K_{\M} \Gamma^d.
\]
The third inequality follows from the simple estimate $\mathrm{vol}(B(\xi,\gamma q))\le \omega_{\M} (\gamma q)^d$.

Summing over $\xi \in\Xi$, we obtain:
\[
\sum_{\xi\in \Xi} I_{\xi}\le
\frac{\omega_{\M}}{3}3^{p}K_{\M}^{p-1}C_2^p   \Gamma^{d(p-1)}\gamma^{d+\epsilon p}q^d \sum_{\xi\in\Xi} \sum_{\d(\zeta,\xi)\le \Gamma q} |s(\zeta)|^p
\le \frac{\omega_{\M}}{3}\left(\mathcal{K} \Gamma^{d}\gamma^{\epsilon} \right)^p(\gamma q)^d
\|s_{|_{\Xi}}\|_{\ell_p(\Xi)}^p,
\]
where $\mathcal{K} := 3 K_{\M}C_2$. 
The final estimate results by exchanging the two summations, and employing the fact that $\#\{\xi \in \Xi\colon \ \d(\zeta,\xi)\le \Gamma q\} \le K_{\M} \Gamma^d$.

Thus, in order to force $\sum_{\xi\in \Xi} I_{\xi}\le \frac{1}{3} \frac{ \alpha_{\M}}{ 3^{p}} (\gamma q)^d\|s_{|_{\Xi}}\|_{\ell_p(\Xi)}^p$, 
the choice of  $\gamma^{\epsilon} < \frac{\alpha_{\M}}{\omega_{\M}}(9 K_{\M}C_2 \Gamma^d )^{-1}$
suffices.

This completes the proof of the lemma.
\end{proof}

\begin{lemma}\label{Inc_Gamma} 
Let $\nu>0$, $d\ge 1$ . For every $\epsilon>0$ there is 
$\Gamma_0\ge(d+1)/\nu$, (depending on $\nu,\,d,\, \epsilon$) so that for $\Gamma\ge \Gamma_0$ and for all $p\in [1,\infty)$,
$$\max\left( \Gamma^{dp}\exp(-\nu p \Gamma),
  \frac{2}{\Gamma^{p}}
  \int_{\Gamma}^{\infty}  
    \exp\bigl(-\nu p r\bigr)  
    r^{(d+1)p-1} 
  \dif r 
\right)
\le 
\epsilon^{p}.$$
\end{lemma}

\begin{proof}
It suffices to establish that
\[
 \frac2{\Gamma^p} \int^\infty_\Gamma \exp(-\nu pr) r^{(d+1)p-1} \text{ d}r \le \epsilon^p
\]
for an appropriate $\Gamma_0$ and $\Gamma\ge \Gamma_0$. To see this, assume $\Gamma_0\ge \max(e, \frac{2(d+1)}\nu)$ which implies
\[
 r \le \frac{\Gamma_0}e e^{r/\Gamma_0}\quad \text{whenever}\quad r\ge \Gamma_0.
\]
Hence
\begin{align*}
 &\frac2{\Gamma^p} \int^\infty_\Gamma \exp(-\nu pr) r^{(d+1)p-1} \text{ d}x\\
\le~ &\frac2{\Gamma^p} \int^\infty_\Gamma \exp(-\nu pr) \left[\frac{\Gamma_0}e e^{r/\Gamma_0}\right]^{(d+1)p-1} \text{ d}r\\
\le~ &\left(\frac{\Gamma_0}e\right)^{dp} \cdot \frac{2e}{\Gamma_0e^p} \int^\infty_\Gamma e^{-\frac{p\nu}2\cdot r} \text{ d}r\\
=~ &\left(\frac{\Gamma_0}e\right)^{dp}\cdot \frac{4e}{\Gamma_0 p\nu e^p} \cdot e^{-p\nu \Gamma/2}.
\end{align*}
The desired conclusion now follows easily.
\end{proof}
%
\subsection{Stability of the Lagrange Basis}
From Propositions \ref{uppercomparison} and \ref{lowercomparison} we now have:
\begin{theorem}
Under Assumptions \ref{LagrangeBasis}, \ref{LagrangeDecay} and \ref{Bernstein},
there exist constants $0<c_1<c_2$, so that
$$
c_1\|A_{p,\cdot}\|_{\ell_p(\Xi)} 
\le 
\|s \|_{L_p(\M)}
\le 
c_2 \|A_{p,\cdot}\|_{\ell_p(\Xi)}.
$$
holds for all $s =\sum_{\xi \in \Xi} A_{\xi} v_{\xi}  \in X$.
\end{theorem}
It follows that the Lagrange basis considered in \cite{HNW} is stable.
\begin{corollary}
Let $\M$ be a $d$-dimensional, compact, Riemannian manifold, let $m>d/2$ 
and let $\kappa_{m,\M}$ be one of the kernels considered in \cite{HNW}.
If $\Xi\subset \M$ has mesh ratio $\rho$, then the associated Lagrange basis
$(\chi_{\xi})_{\xi\in\Xi}$, $\chi_{\xi} = \sum a_{\zeta} \kappa_{m,\M}(\cdot,\zeta)$, satisfies
(\ref{comparison}) for all $s\in S(\kappa_{m,\M})$, with $c_1,c_2$ depending only
on $\rho$, $m$ and $M$.
\end{corollary}
\begin{proof}
The Lagrange functions satisfy Assumption \ref{LagrangeDecay} by \cite[Proposition 4.5]{HNW}.
It remains to show that they are uniformly H\"{o}lder continuous. This is shown in Lemma \ref{Bern_lemma}, which we postpone until the appendix.
\end{proof}
\section{The $L_\infty$-norm of the $L_2$-projector}\label{LeastSquares}

We now discuss the problem of best $L_2$ approximation from finite
subspaces of $\M$, an elementary problem, but one which we consider for
the sake of clarity.
Given a family (indexed over the set $\Ind$, with $\#\Ind<\infty$) 
of continuous functions 
$(v_{\xi})_{\xi\in\Ind}$ on $\M$, we denote  the `synthesis map' by 
 $V \colon \ \comps^{\Ind} \to X := \spam_{\xi\in\Xi} v_{\xi} \subset C(\M)$. 
This is an operator which maps sequences of coefficients to functions in $X$ via
\[
 (a_{\xi})_{\xi\in\Ind}\mapsto \sum_{\xi\in\Ind} a_{\xi} v_{\xi}.
\]
The corresponding `analysis map'
\[
 V^{*}\colon \  L_1(\M) \to \comps^{\Ind}\colon \  s\mapsto \bigl(\langle s,v_{\xi}\rangle\bigr)_{\xi\in \Ind},
\]
is the formal adjoint of the synthesis map. 
Since each $v_{\xi}$ is in $L_1(\M)\cap L_{\infty}(\M)$, 
it follows that 
$V^{*}$ is a bounded operator on each $L_p$ space.

If $V$ is 1--1, it follows that 
the $L_2$ projector on 
$X$ is the operator
$T_{\Ind}  = V(V^* V)^{-1} V^* $, and it, too,
 is bounded on each $L_p(\M), 1\le p\le\infty$. One can write
$T_{\Ind}(f)=\sum_{\xi\in\Ind} c_{\xi} v_{\xi},$ for  $f \in L_p,$
with coefficients $c_{\xi}$ (uniquely) determined
by the Gram matrix $\sfG = V^*V$ and inner products with $f$,
\[
\sfG (c_{\xi})_{\xi\in\Xi} =  V^*f\quad \Leftrightarrow \quad \langle T_{\Ind}(f), v_{\xi} \rangle = \langle f, v_{\xi} \rangle, \quad \xi \in X.
\]
A consequence of this fact is that $(T_{\Xi}f- f) \perp X$ and
$$\|f - T_{\Xi} f\|_2 \le \|f -s\|_2$$ for all $s\in X$.

We make special note of the fact that the $L_2$ projector is independent of the choice of the basis. As such, the formal dependence of $T_{\Xi}$ on $\Xi$ should seem out of place. 
We adopt this notation because when $X= S(\kappa,\Xi)$,
the index set $\Ind \subset \M$ is contained in $\M$, and the spaces (not just the bases) are naturally indexed by $\Xi$. 

The norm of $T_{\Ind}$ on $L_{\infty}$ matches the operator norm on $L_1$, 
and this quantity controls the $L_p$ operator norms for all $1\le p\le\infty.$ Indeed, 
$$\|T_{\Ind}\|_{\infty} = \sup_{\|f\|_{\infty}=1}\sup_{\|g\|_1=1} |\langle T_{\Ind} f, g\rangle| = \sup_{\|f\|_{\infty}=1}\sup_{\|g\|_1=1} |\langle T_{\Ind} g, f\rangle| = \|T_{\Ind}\|_1.$$

We will establish an upper bound for $\|T_{\Ind} \|_{\infty} $ by separately estimating the following three norms:
\[
 \|V\|_{\ell_{\infty}\to L_{\infty}}, \quad \|(V^* V)^{-1}\|_{\ell_{\infty}\to \ell_{\infty}}, \quad \|V^*\|_{L_{\infty}\to \ell_{\infty}}.
\]
The challenging part lies in the process of estimating $\|(V^* V)^{-1}\|_{\ell_{\infty}\to \ell_{\infty}}$.
Under basic assumptions on the basis functions,  we can bound the product of the remaining two operator norms $\|V\|_{\infty}\times\|V^{*}\|_{\infty} $, with little effort.

\subsection{The $\ell_\infty-$boundedness of the inverse of the Gramian}\label{Gram}
Using the basis $(v_{2,\xi})_{\xi \in \Xi}$, renormalized as in Remark \ref{renorm}, the $L_2$ projector $T_\Xi$
has Gram matrix
$\sfG:=(\langle v_{2,\xi}, v_{2,\zeta}\rangle)_{\xi,\zeta}$.
The equation
$$ \frac{\sum_{\xi,\zeta \in \Xi}a_{\xi}\sfG(\xi,\zeta)a_{\zeta}}{\sum_{\xi \in \Xi}|a_{\xi}|^2} = 
\frac{\| \sum_{\xi \in \Xi}a_{\xi}v_{2,\xi}\|_2^2}{\|a\|_{\ell_2(\Xi)}^2}.$$
shows that the {\em Riesz bounds} associated with this basis determine the numerical range of the Gramian. 
By estimate (\ref{comparison}), the eigenvalues of $\sfG$ are bounded below and above by $c_1^2$ and $c_2^2$.

We now investigate $\sfG$ and its inverse. 
For a local basis, as in Assumption \ref{LagrangeDecay}, it is possible to estimate the size of an entry by its distance from the diagonal. 
For $\xi,\zeta\in \Xi$ sufficiently far apart, we write $\Omega_1:=\{x\in\M\mid \d(x,\xi)\le \d(x,\zeta)\}$
and $\Omega_2:=\{x\in\M\mid \d(x,\zeta)\le \d(x,\xi)\}$. Noting that both
$\d(\zeta,\Omega_1)$ and $\d(\xi,\Omega_2)$ are no less than $\d(\xi,\zeta)/2$, 
 (\ref{LagrangeDecay}) implies
\begin{align*}
|\sfG(\xi,\zeta)| &=  \left|\int_{\M} v_{2,\xi}(x)v_{2,\zeta}(x)\dif x \right|
\le
\int_{\Omega_2} |v_{\infty,\xi}(x)| |v_{1,\zeta}(x)|\dif x + \int_{\Omega_1} |v_{\infty,\zeta}(x)| |v_{1,\xi}(x)|\dif x\\
&\le
C_{\sfG} \exp\left[-{\mu}\frac {\min(\d(\xi,\zeta),\injconst)}{q}\right]
\end{align*}
with $\mu := \nu/2$. Here $C_{\sfG}$ is a constant depending only on the manifold $\M$ and the  constants in Assumption \ref{LagrangeDecay}.

To show the $\ell_\infty-$boundedness of the inverse of the Gramian, we adapt an argument developed by Demko, Moss and Smith \cite{DMS}.  Their results primarily apply to bandlimited matrices\footnote{ Demko, Moss and Smith also addressed a few exceptional cases in which the matrices involved are full. However, the discussion there does not seem to imply the situation we are dealing with here.}, and our Gramians are not such matrices. To use their argument, 
we decompose the Gramian as a sum of a bandlimited matrix $\band$ and a negligible matrix $\res$. To be precise, we write
\[
 \sfG = \band+ \res = \band(\mathsf{Id} +\band^{-1}\res),
\]
in which $\band(\xi,\zeta) = 0$ for 
$\d(\zeta,\xi)\ge  \Gamma q$ and $\res(\xi,\zeta) = 0$ 
for
$\d(\zeta,\xi)<  \Gamma q$, where $\Gamma > 1$ 
is a cutoff parameter we need to administer with care in our proof. 
The main challenge in the process of bounding 
$\|\sfG^{-1}\|_{\infty} = \| ( \mathrm{Id} + \band^{-1} \res)^{-1}\band^{-1}\|_{\infty}$ 
is to show that $\|\band^{-1} \res\|_{\infty}<1.$ 
The key idea is that we bound the norm of $\res$ by the integral $\int_{\Gamma}^{\infty} r^d e^{-\mu r} \dif r$, 
which decays exponentially with $\Gamma$ whereas the norm of 
$\band^{-1}$ is comparable to $\Gamma^d$.
%
%
\begin{proposition}\label{GramInverse} Suppose the basis $(v_{\xi})_{\xi\in\Xi}$ is stable and local. I.e., it satisfies 
(\ref{comparison}) and
Assumption \ref{LagrangeDecay}. Then there is a constant $R>0$ so that for all $\Xi$ with
$q<R$, the inverse of Gram matrix 
$\sfG = (\langle v_{2,\xi},v_{2,\zeta}\rangle)_{\xi,\zeta}$ 
has $L_{\infty}$ norm bounded by a constant depending only on $\M$ and
the constants from (\ref{comparison}) and Assumptions \ref{LagrangeDecay}. 
In fact, we have
$$\|\sfG^{-1}\|_{\infty} \le 2 \mathcal{C} \Gamma^d$$
with constants $\mathcal{C}$ and $\Gamma$ (depending only on $\M$, $C_1$, $\nu$ and $c_1$ and $c_2$) as defined in (\ref{Gram_Const}) and (\ref{Gamma_def}) below.
\end{proposition}
%
\begin{proof}
By picking $\Gamma$ sufficiently large and $q$ sufficiently small, 
$\|\res\|_{\infty} \le  \sup_{\zeta} \sum_{\d(\xi,\zeta)\ge \Gamma q} |\sfG(\xi,\zeta)|$
can be made as small as desired. Indeed, we have
\begin{equation*}
 \|\res\|_{\infty}\le \sup_{\zeta} \sum_{\d(\xi,\zeta)\ge \Gamma q} |\sfG(\xi,\zeta)|
\le \sup_{\zeta} \sum_{k=1}^{\infty} \sum_{\xi\in H_k(\zeta)} |\sfG(\xi,\zeta)|
\end{equation*}
We have defined 
$H_k(\zeta) := \{\xi\in\Xi\colon\  k \Gamma q \le d(\xi,\zeta) \le (k+1) \Gamma q\}$
to be the set of centers in an annulus around $\zeta$ of radius $k\Gamma q$.
For small $k,\xi\in H_k(\zeta)$, ensures that $|\sfG(\xi,\zeta)|\le C_\sfG \sum_{\xi\in H_k(\zeta)} \exp(-\mu k\Gamma)$, while for large $k$, (say $k\ge k^*$) $\xi\in H_k(\zeta)$ implies $|\sfG(\xi,\zeta)|\le C_\sfG \sum_{\xi\in H_k(\zeta)} \exp(-\mu \inrad/q)$.
By symmetry, the same inequality holds for $\|\res\|_1$.
Since the cardinality of each set $H_k(\zeta)$ is less than $K_{\M} (k+1)^d \Gamma^d$, while the cardinality of $\bigcup_{k\ge k*} H_k(\zeta)\subset \Xi$ is less than $\left(\frac{\mathrm{diam}(\M)}q\right)^d$, we have
\begin{align}\label{Gamma_def}
\max(\|\res\|_{1},\|\res\|_{\infty})
&\le 
C_{\sfG}K_{\M}\left( \Gamma^d  \sum_{k=1}^{\infty}  (k+1)^d \exp(-\mu k\Gamma) 
+ \left(\frac{\mathrm{diam}(\M)}{ q}\right)^{d} \exp( -\mu \inrad/q)\right)\nonumber\\
&\le 
\mathcal{K}\left(\Gamma^d \exp(-\mu \Gamma) 
+ \frac{1}{\Gamma}\int_{\Gamma}^{\infty}  
  r^{d}e^{-\mu r}
\dif r +
\left(\frac{\mathrm{diam}(\M)}{ q}\right)^{d} \exp( -\mu \inrad/q)\right).
\end{align}
By choosing $\Gamma$ sufficiently large and then $q$ sufficiently small, 
we can impose the condition that the last expression is
less than
$$ \frac12 \min(c_1^2,\mathcal{C}^{-1} \Gamma^{-d}),$$
where  
$\mathcal{C}$ is defined in Equation \eqref{Gram_Const}.

It follows that $\|\res\|_p$ is less than 
$\frac12 c_1^2$ for every $1\le p\le \infty$. 
In particular, this holds when $p=2$,
which implies that $\sigma(\band)$, the spectrum of $\band$, is contained in the interval $[\tfrac12 c_1^2, \tfrac32 c_2^2] \subset [\tfrac12 c_1^2,2c_2^2].$

We are now ready to  use Demko, Moss and Smith's argument.
It is possible to approximate the univariate function $\frac{1}{x}$ on 
$[\tfrac12 c_1^2,2c_2^2]$ using $p_n$, an algebraic polynomial of degree $n$
with uniform error 
$$\max_{x\in[\tfrac12 c_1^2,2c_2^2]}\left|\frac{1}{x} - p_n(x)\right|
\le C_0 Q^{n+1}$$
where
\begin{equation} \label{c-zero}
Q:=\frac{2c_2-c_1}{2c_2+c_1} <1, 
\quad 
\text{and} 
\quad 
C_0:=\frac{(c_1+c_2)^2}{2c_1^2c_2^2}\le \frac{3}{2}c_1^{-2}
\end{equation}

We remark that $\band$ is bandlimited, and that
$\band^j (\xi,\zeta) =0$ for $\d(\xi,\zeta)\ge j\Gamma q$. 
It follows that for $\xi, \zeta$ with $n\Gamma q\le \d(\zeta,\xi)$, 
we have  $p_n(\band)(\xi,\zeta)=0$. 
Thus we can use the spectral theorem for positive operators to write
$$
|\band^{-1}(\xi,\zeta)| 
= 
|\band^{-1}(\xi,\zeta) - p_n(\band)(\xi,\zeta)|
\le 
\|\band^{-1}- p_n(\band)\|_2 
=
\max_{x\in \sigma(\band)}\left|\frac{1}{x} - p_n(x)\right|
\le 
C_0 Q^{n+1}.
$$
By setting 
$$\tau:= - \log Q = - \log \frac{2c_2-c_1}{2c_2+c_1},$$
we derive the following estimate on $|\band^{-1}(\xi,\zeta)|$:
$$
|\band^{-1}(\xi,\zeta)|
\le  
\frac{3}{2}{c_1}^{-2} \exp\left[- \tau \frac{d(\xi,\zeta)}{\Gamma q}\right].
$$
We note that $\tau$ depends only on $c_1$ and $c_2$ 
and not at all on $\Gamma$.

As we did for $\res$, we estimate the
$L_{\infty}$ operator norm of $\band^{-1}$ to obtain
$$
\|\band^{-1}\|_{\infty}
\le
\frac{3}{2}{c_1}^{-2}
  \sum_{k=0}^{\infty}
     K_{\M} (k+1)^{d}e^{-\tau k/\Gamma}
\le  \mathcal{C} \Gamma^d
$$
where 
\begin{equation}\label{Gram_Const}
\mathcal{C} = 3{c_1}^{-2}K_{\M}\tau^{-d}\int_{0}^{\infty} r^d e^{-r}\dif r.
\end{equation}

To complete the proof, we note that
$\|\band^{-1}\res\|_{\infty} \le \mathcal{C}\Gamma^d \times \frac12 \Gamma^{-d}\mathcal{C}^{-1}\le\frac{1}{2}$,
and consequently that $\|\sfG^{-1}\|_{\infty}\le  2 \mathcal{C} \Gamma^d$.
\end{proof}
\begin{remark}
In \cite{GroLei,SchSiva} it is shown that invertible matrices that decay exponentially away from the diagonal
have inverses with the same property. However, it is not clear that their results are applicable to this situation. In particular,  it is unclear how to control the ensuing constants in order to bound the norms of these inverses.
\end{remark}
%
\section{Main Results}\label{Main}
In this section, we show that for a space $X$ with a local, stable basis, 
 the $L_2$ projector is uniformly bounded in each $L_p$ norm.
\begin{theorem}\label{main}
Let $\Xi\subset \M$ be a quasiuniform set. 
Assume that (\ref{comparison}) and Assumption \ref{LagrangeDecay} hold true
for the basis $(v_{\xi})_{\xi\in\Xi}$ of $X$. 
Then 
for all
$1\le p\le \infty$, the $L_p$ operator norm of the $L_2$  projector 
$T_{\Xi} = V (V^* V)^{-1} V^*\colon \  L_1(\M) \to X$ is
bounded by a constant depending only on $\M$, $\rho$ and the constants $\nu$, $C_1$, $c_1$ and $c_2$ that appear in (\ref{comparison}) and Assumption \ref{LagrangeDecay}.
\end{theorem}
\begin{proof}
We note that Proposition \ref{uppercomparison} provides upper bounds for the norms $\|V\|_{\infty}$ and $\|V^*\|_{\infty}$, with which we can easily write down
 \begin{align*}
\|V\|_{\infty} & = \max_x \sum_{\xi\in\Xi} |v_{2,\xi}(x)| = q^{-d/2} \max_x \sum_{\xi\in \Xi} |v_{\infty,\xi}(x)|\le q^{-d/2}c_2,\\
\|V^*\|_{\infty}& = \max_{\xi} \int |v_{2,\xi}(x)|\dif x \le q^{d/2} \max_{\xi} \int |v_{1,\xi}(x)|\dif x\le q^{d/2} c_2.
 \end{align*}
It follows that the product $\|V\|_{\infty}$ and $\|V^*\|_{\infty}$ is bounded
by $c_2^2$.
In Proposition \ref{GramInverse}, we have shown that $\|(V^* V)^{-1}\|_{\infty}$ is bounded. Together, these imply that 
$$\|T_{\Xi}\|_{1}=\|T_{\Xi}\|_{\infty} \le \|V\|_{\infty}\|V^*\|_{\infty}\|(V^* V)^{-1}\|_{\infty}$$ 
Thus by Riesz--Thorin interpolation, the $L_p$ operator norm ($1 \le p \le \infty$) satisfies the same bound.
\end{proof}
An immediate corollary is that for quasi-uniform sets $\Xi$, the $L_2$ projector
$T_{\Xi}$ gives rise to universal near-best approximations.
\begin{corollary}\label{nearbest}
 Under the conditions of the previous theorem, for
$f\in L_p$
$$\| f - T_{\Xi} f\|_p \le (1+ \mathcal{K}) \mathrm{dist}_p\bigl(f, S(\kappa,\Xi)\bigr),$$
where $\mathcal{K}  \le \max\limits_p\|T_\Xi\|_p < \infty$.
\end{corollary}

A key motivation for studying stability of the Lagrange
basis was the observation in \cite{HNW} that the Lagrange
functions associated with the Mat\'{e}rn type kernels 
$\kappa_{m,\M}$ satisfied Assumption \ref{LagrangeDecay}.
In the previous section, it was observed that they satisfy the 
stability relation (\ref{comparison})
as well. 
It follows that Theorem \ref{main} holds in the
cases considered in \cite{HNW}. 
\begin{corollary}
Let $\M$ be a $d$-dimensional, compact, Riemannian manifold, let $m>d/2$ 
and let $\kappa_{m,\M}$ be the kernels considered in \cite{HNW}.
If $\Xi\subset \M$ has mesh ratio $\rho$, then the $L_2$ projector is bounded
on $L_p, 1\le p\le \infty$ by a constant depending only on $\rho, \M$ and $m$,
and, therefore delivers near-best approximation in each $L_p$.
\end{corollary}

\subsection{Example: $L_2$ projection on $\sph^2$}\label{Example}
We now provide the example of the kernel $\kappa_{2,\sph^2}$ that 
satisfies our two main assumptions,
and which, by Theorem \ref{Main},
gives rise to uniformly bounded $L_2$ projectors for quasi-uniform centers.
Furthermore, the $L_p$ approximation orders for this kernel can be given exactly,
 which furnishes the operators $T_{\Xi}$ with a precise error analysis.
 
These facts follow from an analysis carried out in \cite{HNW}, where $\kappa_{2,\sph^2}$ was observed to satisfy Assumption \ref{LagrangeDecay},
a fact which was crucial to yielding bounded Lebesgue constants. A complementary fact, observed in \cite[Section 5]{HNW},
is that $\kappa_{2,\sph^2}$ belongs to the class of SBFs studied in \cite{MNPW}. A consequence of this, is that the kernel networks have universal convergence rates, a fact we consider below.

{\bf Setup:} Let $L_d := \sqrt{\lambda^2_d - \Delta_{\sphere}}$ be the
pseudo-differential operator where $\Delta_{\sphere}$ is the
Laplace--Beltrami operator on $\sphere$ and $\lambda_d = \frac{d-1}2$.
The Bessel potential Sobolev spaces have the norm given by
\[
\|f\|_{H^p_\gamma} := \left\|\sum^\infty_{\ell=0}
  (\ell+\lambda_d)^\gamma P_\ell f\right\|_{L^p(\sphere)}.
\]
We note that when $\gamma\in 2\ints$, $L^\gamma_d$ is a {\em differential}
operator and  $W_p^{\gamma}$, the usual Sobolev space (imported
from $\reals^d$ via a partition of unity and local diffeomorphisms,
or obtained using covariant derivatives) is embedded in $H^p_{\gamma}$  
(this does not hold for other values of $\gamma$ when $p=1,\infty$).

Finally for $\beta=1,2,\ldots$, let $G_\beta$ be the Green's function
for $L^\beta_d$, i.e.,
\[
 L^\beta_d G_\beta(x\cdot y) = \delta_y(x).
\]
More generally let $\phi_\beta := G_\beta + G_\beta *\psi$ where
$\psi$ is an $L^1$ zonal function and let 
$S_{\Xi} :=
\spam \{\phi_\beta(x\cdot \xi)\colon \ \xi\in \Xi\}$ where $\Xi$ has
uniformity measure $ h/q \le c_0$.

{\bf Convergence rates:} The following was given in \cite[Theorem 6.8]{MNPW}
\begin{theorem} \label{rates_full}
Let $1\le p \le \infty,\enskip
\beta> d/p' \quad$ $(1/p + 1/p'=1)$. If $f\in H^p_\beta$, then
with $\phi_\beta$ and $S_{\Xi}$ as given above,
\[
\mathrm{dist}_{p}(f,S_{\Xi}) \le 
C\rho^d h^{\beta}\|f\|_{H^p_\beta}.
\]
\end{theorem}
In particular, we have for $\beta \in 2\ints$ that
\[
\mathrm{dist}_{p}(f,S_{\Xi}) \le 
C\rho^d h^{\beta}\|f\|_{W^\beta_p}.
\]
by the embedding stated above.

One can also treat functions of lower smoothness using a standard 
$K$-functional approach. 
To this end, we assume $\beta \in 2\ints$, ensuring that 
$W_p^\beta \subset H_\beta^p$.
For $0<s\le \beta$ and $1\le p \le \infty$, we define the Besov space $B_{p,\infty}^s$ as the collection of functions in $L_p$  for which the following expression 
$$\|f\|_{B_{p,\infty}^s}:=\sup_{t>0} t^{-s} K(f,t) $$
is finite, where the $K$-functional $K(f,\cdot)\colon \ (0,\infty)\to (0,\infty)$ is defined as
$$K(f,t) := \inf\left\{\|f - g\|_{L_p} + t^{\beta}\|g\|_{W_p^\beta}\colon \  g\in W_p^\beta \right\}.$$
For a complete discussion of Besov spaces, we recommend \cite[1.11, and Chapter 7]{Trieb} and the references therein.
An immediate consequence of the definition is that, for an arbitrary function in $B_{p,\infty}^s$, and ${t}>0$, there is $g_{t}\in W_p^\beta$ simultaneously satisfying
$$
\|g_{t}\|_{W_p^\beta} \le 2 {t}^{s-\beta}\|f\|_{B_{p,\infty}^s}
\quad \text{and} \quad
\|f-g_{t}\|_{L_p} \le 2 {t}^s \|f\|_{B_{p,\infty}^s}.
$$
In the setting of the previous theorem, one can approximate $f$ by 
the kernel approximant to $g_{t}$.
\begin{corollary}\label{rates_lower}
 Let $1\le p \le \infty,\enskip
\beta> d/p',\enskip 0<s<\beta, \quad$ $(1/p + 1/p'=1)$. If $f\in B_{p,\infty}^s$, then
with $\phi_\beta$ and $S_{\Xi}$ as given above,
\[
\mathrm{dist}_{p}(f,S_{\Xi}) \le C (1+\rho^d) h^{s}
\|f\|_{B_{p,\infty}^s}.
\]
\end{corollary}


{\em Example} (The case $m=2$, $\M =\sph^2$.)
In \cite[Section 5]{HNW}, the example $\kappa_{2,\sph^2}$ is considered. We note that
$2 > 1.5 = d/2 +.5$, so by Lemma \ref{Bern_lemma}, its shifts satisfy Assumption \ref{Bernstein}
with $\epsilon = .5$.

By applying a simple symmetry argument, the kernel is shown to be 
 a function of the  the collatitude $\theta \in [0,\pi]$,
$$\kappa_{2,\sph^2}(x,y) = \phi\bigl(\theta \bigr),\quad \text{where} \quad \theta := \cos^{-1}\langle x,y\rangle.$$
I.e., it is {\em zonal}. 
Further analysis shows that it satisfies the fourth order differential equation
$$
\frac1{\sin\theta} (\sin\theta \phi'')'' - 
\frac1{\sin\theta} (\cos \theta\cot \theta \phi')' - \frac1{\sin\theta} (\sin\theta \phi')'+ \phi = 0.  \nonumber
$$
By considering a power series solution of this ODE, selecting the sole solution with
appropriate singularity near $\theta=0$, it is observed that
$\phi = G_{\beta} + G_{\beta}*\psi$ with $\beta=4$.
Consequently, for quasiuniform centers the Lagrange functions associated with this kernel provide a stable basis,  interpolation is stable in $L_{\infty}$,
and $L_2$ minimization is stable in $L_p$ for every $1\le p\le \infty$. Moreover, 
the abstract convergence rates from Theorem \ref{rates_full} and Corollary \ref{rates_lower}, combined with Corollary \ref{nearbest} imply the following:
\begin{theorem}
Let $1\le p \le \infty$ and suppose $\Xi\subset \sph^2$ is quasiuniform, 
with mesh ratio $\rho$. 
If $f\in H^p_4$, then
\[
\|f- T_{\Xi} f\|_p \le 
C h^{4}\|f\|_{H^p_4}.
\]
If $0<s<4$ and $f\in B_{p,\infty}^s$, then
\[
\|f- T_{\Xi} f\|_p \le 
C h^{s}\|f\|_{B_{p,\infty}^s}.
\]
\end{theorem}
\section{Conclusion}\label{Conclusion}

A major practical drawback in using the Lagrange basis is the computational overhead needed to construct it. A fundamental challenge will be to construct local kernel 
bases that satisfy (\ref{comparison}) {\em efficiently}.
Another, mainly theoretical, challenge is to find kernels satisfying the two basic assumptions of this
paper for which precise $L_p$ error estimates are known, as was the case in the previous section for 
$\kappa_{2,\sph^2}$. 
A third issue is to find kernels having explicit, closed form representations (unlike the case $\kappa_{2,\sph^2}$, which is given as the solution of a fourth order ODE).

These last two points lead to the following conjectures

{\bf 1. Kernels on  spheres:}
The first conjecture deals with approximation by restrictions to spheres
of the surface splines. I.e., these are kernels used on $\reals^d$ restricted to $\sph^{d-1}$. Earlier analytic properties of these have been studied in
\cite{HL,NSW,CF}. The {\em restricted surface splines} are known to be {\em conditionally positive definite}.

For centers $\Xi\subset \sph^d$, define kernel spaces
$$S(\varphi_m,\Xi):= \left\{\sum_{\xi\in\Xi} A_{\xi} \varphi_m(\cdot,\xi) + P\mid
P\in \Pi_{\lceil m-d/2\rceil}\, \text{and}\,
\sum_{\xi\in \Xi} A_{\xi} Q(\xi) = 0\,,\forall Q\in \Pi_{\lceil m-d/2\rceil}\right\} $$
 generated by shifts of the {restricted surface splines}
$$\varphi_m(x,y) :=
\begin{cases} 
  (1-\langle x, y\rangle)^{m-d/2}\log (1-\langle x, y\rangle) &\quad d\, \text{even}\\
  (1-\langle x, y\rangle)^{m-d/2} &\quad d\, \text{odd}
\end{cases},$$
and spherical harmonics of degree $\lceil m-d/2\rceil$ or less. 
Such kernels are known to be of the form $G_{\beta} +G_{\beta}*\psi$, with $\beta=2m$ \cite[(3.6)]{MNPW} and thus have a precise, theoretical error analysis (similar to the case of $\kappa_{2,\sph^2}$ 
of the previous section).
They also invert an elliptic differential operator of the form $\prod_{j=1}^m (\Delta - r_j)$, with $r_1,\dots,r_m$ a sequence of real numbers \cite[Lemma 3.4]{Hsphere}, as such, they are very similar to the kernels $\kappa_{m,\sph^d},$ which
also invert elliptic differential operators of order $2m$. 
\begin{conjecture*}
If $\Xi$ is quasiuniform with mesh ratio $\rho$, the Lagrange functions $\chi_{\xi}\in S(\varphi_m,\Xi)$ are a local basis (in the sense of Assumption \ref{LagrangeDecay}) and 
satisfy the H\"{o}lder continuity assumption, Assumption \ref{Bernstein}. 
\end{conjecture*}
The validity of this conjecture would have several immediate consequences:
\begin{itemize}
 \item[$\lbrack$A$\rbrack$]  The associated Lebesgue constants for interpolation are bounded by a constant
depending only on $\rho$ (and not on $\# \Xi$),
\item[$\lbrack$B$\rbrack$] Interpolation provides precise $L_{\infty}$ approximation rates. For 
functions $f\in C^s$, $s\le 2m$, one has
$$\|f - I_{\Xi}f\|_{\infty} = \mathcal{O}(h^s).$$
\item[$\lbrack$C$\rbrack$] The Lagrange basis satisfies the stability condition (\ref{comparison}), 
\item[$\lbrack$D$\rbrack$] The $L_p$ norms of the $L_2$ projector $T_{\Xi}$ associated with these spaces are bounded with constants depending only on $\rho$.
\item[$\lbrack$E$\rbrack$] The $L_2$ projector provides precise $L_p$ approximation rates. I.e., approximation order $s$ for functions having $L_p$ smoothness $s$, for $s\le 2m$ (as determined by
inclusion in Bessel potential $H_{2m}^p$ or Besov spaces $B_{p,\infty}^s$):
$$\|f - T_{\Xi}f\|_p = \mathcal{O}( h^s)$$
\end{itemize}

A more ambitious problem would be to prove the previous for all kernels
of the form 
$$\phi_{\beta}=G_{\beta} +G_{\beta}*\psi,\quad \beta>d,$$ 
which includes the compactly supported SBFs, as well as most other SBFs of finite smoothness.
\begin{conjecture*}
If $\Xi$ is quasiuniform with mesh ratio $\rho$, the Lagrange functions 
$\chi_{\xi}\in S(\phi_\beta,\Xi)$ are a local basis (in the sense of Assumption \ref{LagrangeDecay})
and satisfy Assumption \ref{Bernstein}. 
\end{conjecture*}
As a corollary, [A]--[E] would hold for  kernels $\phi_{\beta}$ as well.

{\bf 2. Surface splines on ${SO(3)}$:} The second conjecture deals with 
an analogous problem on $SO(3)$.
In this case, we consider the family of kernels
$$k_m(x,y) :=
  \left(\sin (\omega\bigl(y^{-1} x)/2\bigr)\right)^{m-3/2}.
$$
where $\omega(z)$ is the rotational angle of $z\in SO(3)$ (the expression
$\omega(y^{-1} x)$ is the geodesic distance between rotations $x$ and $y$).
For centers $\Xi\subset SO(3)$, define kernel spaces
$$S(k_m,\Xi):= \left\{\sum_{\xi\in\Xi} A_{\xi} \phi_m(\cdot,\xi) + P\mid
P\in \Pi_{\lceil m-3/2\rceil}\, \text{and}\,
\sum_{\xi\in \Xi} A_{\xi} Q(\xi) = 0,\,
\forall Q\in \Pi_{\lceil m-3/2\rceil}\right\} $$
where  $\Pi_{\lceil m-3/2\rceil}$ is the space of Wigner $D$-functions of degree $\lceil m-3/2\rceil$ or less. 
As in the previous case, they invert a differential operator of the form $\prod_{j=1}^m (\Delta - r_j)$, \cite[Lemma 3]{HS}, and have theoretical, precise error estimates. 
\begin{conjecture*}
If $\Xi\subset SO(3)$ is quasiuniform with mesh ratio $\rho$, the Lagrange functions $\chi_{\xi}\in S(k_m,\Xi)$ are a local basis (in the sense of Assumption \ref{LagrangeDecay}), and
  satisfy Assumption \ref{Bernstein}.
\end{conjecture*}
In this case, points [A]--[E] also would follow.
\section[Appendix]{Appendix}\label{Appendix}
In this section we relate the size of a H\"{o}lder seminorm of a function to 
its norm in a Sobolev space. 
It is easy enough to state and prove in the Euclidean setting.
\begin{lemma}\label{EucZeros}
Let $\epsilon>0$, $m>d/2 + \epsilon$, $r>0$ and suppose the finite set of points 
$\Xi\subset B(0,r)$ is sufficiently dense 
so that $h := \max_{x\in B(0,r)} \d(x,\Xi) < h_0 r$ (with $h_0:= 1/(32m^2)$).
Then there is a constant so that for every $f\in W_2^m \bigl(B(0,r)\bigr)$ and $y\in B(0,r)$ 
$$|f(y) - f(0)| \le C |y|^{\epsilon} \, r^{m-d/2- \epsilon} 
\left( \sum_{|\alpha| = m} \int_{B(0,r)} \left| D^{\alpha} f (x)\right|^2 \dif x\right)^{1/2}.
$$
\end{lemma}
\begin{proof}
Suppose $r=1$. By the Sobolev embedding theorem, there is a constant so that
$$\frac{|u(Y) - u(0)|}{|Y|^{\epsilon} } \le C 
\left( \sum_{|\alpha| \le m} \int_{B(0,1)} \left| D^{\alpha} u(X)\right|^2 \dif X\right)^{1/2}
$$
for all $u\in W_2^m(B(0,1))$. By rescaling this with $rY=y$ and $f(y) = u(Y)$, we 
have 
$$\frac{|f(y) - f(0)|}{|y/r|^{\epsilon} } \le C 
\left( \sum_{|\alpha| \le m} r^{2|\alpha|-d} \int_{B(0,r)} \left| D^{\alpha} f(x)\right|^2 \dif x\right)^{1/2}
$$
for all $f\in W_2^m(B(0,r))$. We can now apply \cite[Lemma 3.7]{HNW} (a standard Bramble-Hilbert type argument) to  obtain
$$ C 
\left( \sum_{|\alpha| \le m} r^{2|\alpha|-d} \int_{B(0,r)} \left| D^{\alpha} f(x)\right|^2 \dif x\right)^{1/2}
\le Cr^{m -d/2} \left( \sum_{|\alpha| = m} \int_{B(0,r)} \left| D^{\alpha} f(x)\right|^2 \dif x\right)^{1/2}
$$
which proves the lemma.
\end{proof}

Before proceeding, we need some necessary geometric background, which can be
found in \cite{HNW}.

\paragraph{Exponential map} We define the \emph{exponential map} 
$\Exp_p\colon \  \reals^d \to \M$, 
our basic diffeomorphism,  by letting 
$\Exp_p(0)=p$ and $\Exp_p(s\bft_p)=\gamma_p(s)$, 
where $\gamma_p(s)$ is the unique geodesic that passes through $p$ for $s=0$ 
and has a tangent vector $\dot \gamma_p(0)=\bft_p$ of length 1.

\paragraph{Metric Equivalence} The family of exponential maps are uniformly isomorphic; i.e., there are constants $0<\Gamma_1 \le \Gamma_2<\infty$ so that
for every $p_0\in \M$ and every $x,y\in B(0,\inj)$
\begin{equation}\label{isometry}
 \Gamma_1|x-y|\le \d(\Exp_{p_0}(x), \Exp_{p_0}(y)) \le \Gamma_2 |x-y|.
\end{equation}

\paragraph{Sobolev spaces}
Let $\Omega\subset \man$ be a measurable subset. 
We define the Sobolev space $W_2^m(\Omega)$ to be the space of all functions $f\colon \ \M \to \reals$ 
so that the real valued function 
$p\mapsto | \nabla^k f |_{g,p} =\langle \nabla^m f , \nabla^m f \rangle_{g,p}^{1/2} $ defined for 
$p\in\Omega$ is in $L_2$. The associated norms are:
%
%
\begin{equation}\label{def_spn} 
\|f\|_{W_2^m(\Omega)}:= 
\left( 
  \sum_{k=0}^m
  \int_{\Omega} 
    | \nabla^k f |_{g,p}^2
  \, \dif \mu(p)\right)^{1/2}.
\end{equation}
For the precise meaning of $\langle \nabla^m f , \nabla^m f \rangle_{g,p}^{1/2}$, we refer the reader
to \cite[Section 2]{HNW}, although we note that in $\reals^d$, its exact meaning
is 
$$\sum_{|\alpha|=k} \begin{pmatrix} k\\ \alpha \end{pmatrix} D^{\alpha} f(p) D^{\alpha} g(p),$$
i.e, the inner product of the $k$th order derivative tensor at $p$, in which case (\ref{def_spn})
is the usual Euclidean Sobolev norm.

\paragraph{Metric equivalence for Sobolev spaces}
The exponential map also gives rise to the following metric equivalence, proved in \cite[Lemma]{HNW}, between Sobolev spaces on regions in the manifold and Sobolev spaces on regions in $\reals^d$.
For $m\in\nats$ and $0<r<\inj$,  there are constants $0 < c <C$ so that
for any measurable 
$\Omega \subset B_{r}$, for all $j\in \nats$, $j\le m$, and for any $p_0\in \M$, 
\begin{equation}\label{SobEquiv}
c \| u\circ\Exp_{p_0}\|_{W_\sfp^j(\Omega)}\le\|u\|_{W_\sfp^j(\Exp_{p_0}(\Omega))}\le C\| u\circ\Exp_{p_0}\|_{W_\sfp^j(\Omega)}
\end{equation}

We are now able to state the analogous zeros lemma for manifolds:
\begin{lemma}[Zeros Lemma]\label{zeros}
Let $\epsilon>0$, let $m$ be a positive integer, greater than $d/2+\epsilon$, 
and let $r$ be a positive real number less than $\inj$, the injectivity radius of $\M$. 
Suppose that $\Xi \subset B(x,r)\subset \M$
is a discrete set with $h \le \Gamma_1 h_0 r$. If 
$u \in W^{m}_2 (B(x,r))$ satisfies $u_{|_{\Xi}} = 0$, then for every $x,y\in B(x,r)$,
$$|u(x)-u(z)| \le C r^{m-\epsilon-d/2}\d(x,z)^{\epsilon}\|u\|_{W^{m}_2(B(x,r))},$$ 
where $C$ is a constant independent of $u,x,z,h$ and $r$, and where $\Gamma_1$ depends only on $m$ and $\M$.
\end{lemma}
\begin{proof}
This follows from  (\ref{isometry}), (\ref{SobEquiv}) and Lemma \ref{EucZeros},
(which we can apply since $h\le \rad r$ implies that the $\Exp_x (\Xi)$ has fill distance
less than $h_0 r$)
  in precisely the same way as \cite[Lemma 3.8]{HNW}. Letting $z = \Exp_x(y)$
$$|u(x) - u(z)|= |\widetilde{u}(0)- \widetilde{u}(y)|
\le r^{m-d/2-\epsilon/2} |y|^{\epsilon} \|\widetilde{u}\|_{W_2^m(B(0,r))}\le Cr^{m-d/2-\epsilon} [\d(x,z)]^{\epsilon} \|u\|_{W_2^m(B(x,r))}.$$
We note that $|y| = |y-0|= \d(\Exp_x(y),\Exp_x(0))=\d(z,x)$.
\end{proof}
We now show that the Lagrange functions associated with the kernels $\kappa_{m,\M}$
satisfy the H\"{o}lder continuity property, Assumption \ref{Bernstein}.
\begin{lemma}\label{Bern_lemma}
Let $0<\epsilon\le 1$. For $m>d/2+\epsilon$ and for $\Xi\subset \M$ having mesh ratio $\rho$,
there is a constant $C$ depending only on $m$, $\M$, $\epsilon$ and $\rho$ for which 
$$|\chi_{\xi}(x) - \chi_{\xi}(y)|\le C \left[\frac{\d(x,y)}{q}\right]^{\epsilon}.$$
\end{lemma}
\begin{proof}[Proof of Lemma \ref{Bern_lemma}]
By compactness and continuity of the functions $s$, it suffices to consider centers $\Xi$ 
 and pairs $x,y$ that satisfy $\d(x,y)\le h\le\Gamma h_0$, a quantity appearing in the zeros lemma, Lemma \ref{zeros} 
 ($h_0$ is defined in Lemma \ref{EucZeros} and $\Gamma_1$ is the lower bound in (\ref{isometry})).

By \cite[Corollary 4.4]{HNW}
there is a constant $\nu>0$ such that, then for $0\le T<\inj$ we have the estimate 
\begin{equation*}
\|\chi_{\xi}\|_{W_2^m\bigl(\M \setminus B(\xi,T)\bigr)} 
\le  C_{m,\M} \, e^{-\nu(\frac{T}{h})} \, q^{d/2-m}
\end{equation*}
where $q$ is the minimal separation distance between points of $\Xi$.
We proceed as in the proof of \cite[Proposition 4.5]{HNW} by  considering three cases.

%
{\bf Case 1.} For  $\d(\xi,x)\le \rad$, we can apply Lemma \ref{zeros} with $r=2\rad$
(i.e., on the ball $B(x,2\rad)$, where the zeros of $\chi_{\xi}$ have mesh norm bounded by $2h$)
to achieve
$$
\left|\frac{\chi_{\xi}(x)-\chi_{\xi}(y)}{[\d(x,y)]^{\epsilon}}\right| \le C
\left(2\rad\right)^{m-\epsilon-d/2}
\|\chi_{\xi}\|_{m, B(x,2\rad)} 
\le
C
\left(\frac{h}{q}\right)^{m -\epsilon-d/2}q^{-\epsilon}.$$
 Thus, for $x \in B(\xi,\rad)$ we have
$ |\chi_{\xi}(x) - \chi_{\xi}(y)| \le C [\d(x,y)]^{\epsilon}q^{-\epsilon} \exp(-\nu \frac{\d(\xi,x)}{h})$.

%
{\bf Case 2.} A similar argument holds for $\rad<T=\d(\xi,x)\le \inj$, where we obtain
 \begin{align*}
  |\chi_{\xi}(x) - \chi_{\xi}(y)| 
&\le C h^{m -\epsilon-d/2} [\d(x,y)]^{\epsilon} \|\chi_{\xi}\|_{W_2^m\bigl(\M\setminus B(\xi, T - \rad)\bigr)}\\
&\le C[\d(x,y)]^{\epsilon}  q^{-\epsilon}\, \exp\left[-\nu\left(\frac{\d(x,\xi)}{h}\right)\right].
\end{align*}
In the second inequality we apply  \cite[Corollary 4.4]{HNW} after applying Lemma \ref{zeros} and making the observation
that $B(x,\rad) \subset \M\setminus B(\xi,T-\rad)$.

%
{\bf Case 3.} For $T>\inj$, we make use of the fact that 
$\M\setminus B(\xi,\inj)\supset\M \setminus B(\xi,T)$, so
we simply use the estimate for $T=\inj$ to obtain
$$|\chi_{\xi}(x) - \chi_{\xi}(y)| \le C  [\d(x,y)]^{\epsilon}q^{-\epsilon} \exp\left[-\nu\left(\frac{\inj}{h}\right)\right]
.$$
This completes the proof of the lemma.
\end{proof}
\bibliographystyle{siam}
\bibliography{HNSW}
\end{document}